\documentclass{amsart}
\usepackage{latexsym,amsmath,amsthm,amssymb}

\newtheorem{theorem}{Theorem}[section]
\newtheorem{lemma}[theorem]{Lemma}
\newtheorem{corollary}[theorem]{Corollary}
\newtheorem{proposition}[theorem]{Proposition}

\theoremstyle{remark}

\newtheorem{remark}[theorem]{Remark}

\theoremstyle{definition}

\newcommand{\dR}{\ensuremath{\mathbb{R}}} 
\newcommand{\R}{\dR}

\begin{document}

\title{On Sobolev regularity of mass transport and transportation inequalities}

\maketitle

\centerline{\large \footnote{ Moscow State University of Printing Arts, St. Tikhon Orthodox University, and Higher School of Economics (Moscow)}Alexander V. Kolesnikov}

\begin{abstract}
We study Sobolev a priori estimates for the optimal transportation $T = \nabla \Phi$ between  probability measures $\mu=e^{-V} \ dx$ and $\nu=e^{-W} \ dx$ on $\R^d$.
Assuming uniform convexity of the potential $W$ we show that $\int \| D^2 \Phi\|^2_{HS} \ d\mu$, where $\|\cdot\|_{HS}$ is the Hilbert-Schmidt norm,
is controlled by the Fisher information of $\mu$. In addition, we prove similar estimate for the $L^p(\mu)$-norms of $\|D^2 \Phi\|$ and obtain some $L^p$-generalizations of the well-known Caffarelli 
contraction theorem.
We establish a connection of our results with the Talagrand transportation inequality.
We  also prove a corresponding dimension-free version for the relative Fisher information with respect to a Gaussian measure.
\end{abstract}


\section{Introduction}
Let $\mu = e^{-V} dx$ and $\nu= e^{-W} dx$ be  probability measures on $\R^d$
and let $T=\nabla \Phi$ be the optimal transportation mapping  such that $\nu$ is the image of $\mu$ with respect to $T$: $\nu = \mu \circ T^{-1}$. In what follows
we say for brevity that $T$   sends (pushes forward)   $\mu$ onto $\nu$. The corresponding convex potential is denoted by $\Phi$.
The reader is advised to consult \cite{Vill} for an account in the optimal transportation theory.

Assuming that $W$ is uniformly convex ($D^2 W \ge K \cdot \mbox{Id}, \ K > 0$ )  we prove that
\begin{equation}
\label{main}
\mathcal{I}_{\mu} := \int |\nabla V|^2 \ d\mu \ge K \int \|D^2 \Phi\|^2_{HS} \ d\mu.
\end{equation}
More generally, we show that for every unit $e \in \R^d$ and $p \ge 1$
\begin{equation}
\label{lp-phi}
\frac{p+1}{2} \| V^2_e \|_{L^p(\mu)} \ge K  \| \Phi^2_{ee} \|_{L^p(\mu)}.
\end{equation}
These results can be considered as (global, dimension-free)  Sobolev  a priori  estimates for the following Monge-Amp{\`e}re
equation
$$
e^{-V} = e^{-W(\nabla \Phi)} \det D^2 \Phi.
$$
The regularity theory for the Monge-Amp{\`e}re operator has a quite long history. Many famous scientists contributed to this  area.
We advise the reader to consult  \cite{Guit}  (see also \cite{Bak}, \cite{Pog}, \cite{GT},  \cite{CafCab}, \cite{Kryl2}, \cite{Vill}).
In particular, some Sobolev a priori estimates for the optimal transportation have been obtained by L. Caffarelli in  \cite{Caf-90}. The most recent results in this direction are concerned
with the H{\"o}lder regularity of optimal transportation  maps on manifolds
(see \cite{TW}, \cite{Loep}, \cite{DL}, \cite{KimMC}, \cite{FKMC} and the references therein).

The approach we use here is in a sense  probabilistic. The estimates obtained in this paper are
 1) dimension-free, 2) global, 3) can be obtained in a  constructive way by integration-by-parts and above-tangential formalism.
We  refer to  the works of N.~Ivochkina (for instance, \cite{Ivoch}) for some similar arguments.
In spite of the large amount of results, the only global dimension-free estimate known before was given by the Caffarelli
contraction theorem \cite{Caf}.  According to this result  every optimal transportation $T$ sending the standard Gaussian measure onto  a log-concave measure $\nu$ with uniformly convex $W$ (i.e. $D^2 W \ge K \cdot \mbox{Id}$  with $K>0$)
 is a $\frac{1}{\sqrt{K}}$ - contraction (i.e. $\|T\|_{Lip} \le \frac{1}{\sqrt{K}}$).

This contraction theorem has become very popular among  probabilists because it gives immediately very nice analytical consequences (for instance, the Bakry-Ledoux theorem, a probabilistic version of the L{\'e}vy-Gromov
comparison theorem). Some recent generalizations  can be found in \cite{Kol}, \cite{KimMilman}, \cite{Vald}.
Another applicatons are: log-Sobolev and isoperimetric inequalities. 
By a recent observation of E. Milman (see \cite{Milman1}, \cite{Milman2}), even weaker $L^p$-estimates for $\| D^2 \Phi\|$
imply results of this type if the image measure is log-concave.

We note (though it is not aim of this paper) that in this way one can also establish some 
Sobolev estimates for the third-order derivatives.
  Our estimates rely on the following (formal) identity:
$$
\int    V^2_{x_i} \  d \mu
=
 \int \langle D^2 \Phi \cdot D^2 W(\nabla \Phi) \cdot D^2 \Phi \cdot e_i, e_i \rangle \ d \mu
+
\int \bigl\| (D^2 \Phi)^{-\frac{1}{2}} D^2 \Phi_{x_i}  (D^2 \Phi)^{-\frac{1}{2}} \bigr\|^2_{\rm HS} \ d\mu .
$$
In particular, if $\Phi$ is sufficiently smooth and  $D^2 W \ge K \cdot \mbox{Id}$, $K>0$, then this identity implies (\ref{main}) and the following estimate for the third-order derivatives
of $\Phi$:
$$
\int    |\nabla V|^2 \  d \mu  \ge 2 \sqrt{K}  \int  \Bigl[\sum_{i=1}^d \|D^2 \Phi_{x_i} \|^2_{HS} \Bigr]^{\frac{1}{2}} \ \ d\mu .
$$

Another motivation for this study  comes from the probability theory.
It's worth noting that (\ref{main}) appears to be very similar to the well-known Talagrand inequality (see \cite{Tal}),
which is  a classical representative of the so-called transportation inequalities (see surveys \cite{Led}, \cite{GL}),
 close relatives of  various functional inequalities
(concentration, Sobolev, isoperimetric, etc.).
Let $\gamma$ be the standard Gaussian measure.
 Consider the optimal transportation $\nabla \Phi$ of
$g \cdot \gamma$ onto $\gamma$. Then the following (Talagrand or transportation inequality) holds
\begin{equation}
\label{Tal}
\mbox{Ent}_{\gamma} g \ge \frac{1}{2} W^2_2(\gamma, g \cdot \gamma ) , \  \  \mbox{where}
\end{equation}
$$
\mbox{Ent}_{\gamma} g = \int g \log g \ d\gamma, \   \   \   W_2(\gamma, g \cdot \gamma ) = \Bigl(\int |x - \nabla \Phi(x)|^2   g
\ d \gamma\Bigr)^{1/2}
$$
are the relative entropy and the Kantorovich distance.

We recall that the Talagrand inequality follows from the so-called displacement convexity property of the entropy functional (see \cite{AGS}, \cite{Vill}).
Note in this respect that the energies (Fisher information etc.), unlike entropies,  are NOT displacement convex.
Nevertheless, in Section 3 we reveal a direct relation of (\ref{main}) to (\ref{Tal}).
First we prove the inequality
\begin{equation}
\label{inc-ineq}
\int (V(x+ e) - V(x))  \ d\mu
\ge
\frac{K}{2} \int  |\nabla \Phi(x+e)- \nabla \Phi(x)|^2 \ d\mu,
\end{equation}
where $e \in \R^d$.  It turns out that (\ref{inc-ineq}) can be considered as a version of a generalized Talagrand-type inequality proved in
\cite{Kol04}.
Then we show that (\ref{main}) follows from (\ref{inc-ineq}) under a natural limiting procedure.

In Section 5 we prove some dimension-free estimates of the type (\ref{main}).
For instance, if $\mu = g \cdot \gamma$  (with smooth $g$) and $\nu = \gamma$, then
\begin{align*}
 \mbox{\rm I}_{\gamma} g  & =   2 \mbox{\rm Ent}_{\gamma} g
- 2 \int \log {\det}_2 (D^2 \Phi - \mbox{\rm Id})\ g d \gamma
\\&
+
 \int \|D^2 \Phi -\mbox{\rm Id} \|^2_{HS} \  g d \gamma
 +
\sum_{k=1}^{d} \int \mbox{\rm Tr} \bigl[ (D^2 \Phi)^{-1} D^2 \Phi_{x_k} \bigr]^2 \ g d \gamma,
\end{align*}
where
$ \mbox{\rm I}_{\gamma} g =
\int \frac{|\nabla g|^2}{g} d \gamma$ (relative information),
${\det}_2 (D^2 \Phi - \mbox{\rm Id}) = \det D^2 \Phi \cdot \exp \bigl(d - \Delta \Phi \bigr)$
(the Fredholm-Carleman determinant of  $D^2 \Phi - \mbox{\rm Id}$).

We note that all the terms in the right-hand side are non-negative.
In particular, this identity implies the following stronger version of the log-Sobolev inequality
$$
\mbox{\rm I}_{\gamma} g \ge 2 \mbox{\rm Ent}_{\gamma} g - \int \log {\det}_2 (D^2  \Phi)^2 \  g \ d \gamma 
$$
and the following (essentially infinite-dimensional) analog of (\ref{main})
$$
 \mbox{\rm I}_{\gamma} g  \ge
 \int \|D^2 \Phi -\mbox{\rm Id} \|^2_{HS} \  g d \gamma.
$$
Note that the result stated in this form looks particularly relevant to  the Talagrand inequality.
See also  Remark \ref{extrem} below on uniqueness of the extremals for the classical  log-Sobolev inequality.
In addition, we  prove some dimension-free results for the general log-concave reference measures.

In Section 6 we prove several $L^p$-generalizations of the main result.
We prove that for every fixed unit vector $e$ and $p \ge 1$ one has
$$
K \| \Phi^2_{ee} \|_{L^{p}(\mu)} \le \| (V_{ee})_{+} \|_{L^{p}(\mu)},
$$
$$
K \| \Phi^2_{ee} \|_{L^{p}(\mu)} \le \frac{p+1}{2}  \| V^2_{e} \|_{L^{p}(\mu)}.
$$
We emphasize that all these estimates can be obtained without any use of regularity theory. Instead of it we apply
the change of variables formula from \cite{McCann2} and the above-tangential formalism.
Note that the contraction theorem follows from these estimates and this is exactly the case when $p=\infty$.
In addition, in Section 7 we prove  the following dimension-free estimate for the operator norm $\|D^2 \Phi\|$ 
$$
K 
\Bigl( \int  \| D^2 \Phi\|^{2p} \ d \mu \Bigr)^{\frac{1}{p}}
\le \Bigl(  \int  \| (D^2 V)_{+} \|^p  \ d \mu \Bigr)^{\frac{1}{p}}.
$$

Finally, we note that some of our results  hold not only for the optimal transportation  mappings.
For instance, they  can be established for the so-called triangular mappings (see \cite{B2006}, \cite{ZhOv}).
 See Section 2 and the  forthcoming paper  \cite{KR}.

The author thanks Luigi Ambrosio, Max-Konstantin von Renesse, Michel Ledoux, Emanuel Milman, and Frank Morgan for their interest and stimulating discussions.
This work was partially done during the author's visit to the Technische Universit{\"a}t Berlin under the support
of the German Academic Exchange Service (DAAD).

\section{Heuristic proof}

In this section we give a {\it formal} computation of the main formula of our work.
See  Sections 3 and 4 for rigorous justifications.

In what follows we denote by
$\mathcal{I}_{\mu}$ the Fisher information of $\mu$:
$$
\mathcal{I}_{\mu} =
 \int |\nabla V|^2 \ d\mu
$$
and by $\| A \|_{HS} = \sqrt{\mbox{Tr} (A \cdot A^{T}})$ the Hilbert-Schmidt norm of a matrix $A$.
For the operator norm we use the standard notation $\|\cdot\|$.
It will be assumed throughout that $\mathcal{I}_{\mu} < \infty$ and that $\mu$ and $\nu$ admit the finite second moments.
The last condition is automatically satisfied for $\nu$ if $D^2 W \ge K \cdot \mbox{Id}, \  K >0$.

Let $T$ be a  mapping sending $\mu$ onto $\nu$. We assume that the potentials $V, W$ are smooth, $T : \R^d \to \R^d$ is a smooth diffeomorfism
satisfying $\det DT > 0$. By the change of variables formula
$$
e^{-V} = e^{-W(T)} \det DT.
$$
Taking the logarithm we obtain
\begin{equation}
\label{log-ch-v}
V = W(T) - \log \det DT.
\end{equation}
Choose a unit vector $e$ and differentiate (\ref{log-ch-v})  along $e$ twice. To this end we apply the following fundamental relation
$$
\partial_e \log \det DT = \mbox{Tr} \big[ DT_e \cdot (DT)^{-1} \bigr].
$$
Differentiating once again and applying
$$
DT_e \cdot (DT)^{-1} + DT \cdot \bigl[ (DT)^{-1}\bigr]_e =0
$$
we get
$$
\partial_{ee} \log \det DT = \mbox{Tr} \big[ DT_{ee} \cdot (DT)^{-1} \bigr]
- \mbox{Tr} \bigl[ DT_e \cdot (DT)^{-1} \bigr]^2.
$$
Coming back to (\ref{log-ch-v}) one gets
$$
V_e = \langle \nabla W(T), T_e \rangle
-\mbox{Tr} \big[ DT_e \cdot (DT)^{-1} \bigr],
$$
\begin{equation}
\label{vee}
V_{ee}  = \langle D^2 W(T) \cdot T_e, T_e \rangle
+  \langle \nabla W(T), T_{ee} \rangle
-\mbox{Tr} \big[ DT_{ee} \cdot (DT)^{-1} \bigr]
+ \mbox{Tr} \bigl[ DT_e \cdot (DT)^{-1} \bigr]^2.
\end{equation}
Let us integrate (\ref{vee}) over $\mu$. Clearly, $\int V_{ee} \ d \mu = \int V^2_{e} \ d \mu $.
Let us show that after taking the integral  the terms in the middle cancel each other. Indeed, let us denote $S= T^{-1}$. One has
\begin{align}
\label{MAIBP}
\int \langle \nabla W(T), T_{ee}\rangle \ d \mu  & =
\int \langle \nabla W, T_{ee}(S) \rangle \ d \nu =
\int \mbox{Tr} D \bigl[ T_{ee} (S) \bigr] \ d\nu \\& \nonumber
=  \int \mbox{Tr} \bigl[ DT_{ee} (S) \cdot DS  \bigr] \ d\nu
=  \int \mbox{Tr} \bigl[ DT_{ee} (S) \cdot (DT)^{-1}(S)  \bigr] \ d\nu
\\ &
=  \int \mbox{Tr} \bigl[ DT_{ee} \cdot (DT)^{-1}  \bigr] \ d\mu.
\end{align}
Thus we get
$$
\int V^2_e \ d\mu
=
\int \langle D^2 W(T) \cdot DT \cdot e, DT \cdot e \rangle  \ d\mu
+
\int  \mbox{Tr} \bigl[ DT_e \cdot (DT)^{-1} \bigr]^2 \ d\mu.
$$

We are interested in two particular cases

1) Optimal transportation mappings.

Optimal transportation mappings have the form $T = \nabla \Phi$, where $\Phi$ is the convex function.
In this case one has
\begin{equation}
\label{Fisher-i}
\int    V^2_{x_i} \  d \mu
=
 \int \langle D^2 \Phi \cdot D^2 W(\nabla \Phi) \cdot D^2 \Phi \cdot e_i, e_i \rangle \ d \mu
+
\int \mbox{\rm Tr} \Bigl[ (D^2 \Phi)^{-1} D^2 \Phi_{x_i} \Bigr]^2 \   d\mu.
\end{equation}
Note that the last integrand is non-negative and admits another representation
   $$
 \mbox{\rm Tr} \Bigl[ (D^2 \Phi)^{-1} D^2 \Phi_{x_i} \Bigr]^2
=
\bigl\| (D^2 \Phi)^{-\frac{1}{2}} D^2 \Phi_{x_i}  (D^2 \Phi)^{-\frac{1}{2}} \bigr\|^2_{\rm HS}  .
$$
Taking the sum over $i$ we get
\begin{equation}
\label{Fisher}
\mathcal{I}_{\mu}
=
\int \mbox{\rm Tr} \Bigl[D^2 \Phi \cdot D^2 W(\nabla \Phi) \cdot D^2 \Phi  \Bigr]\ d\mu
+ \sum_{i=1}^{d}
\int \bigl\| (D^2 \Phi)^{-\frac{1}{2}} D^2 \Phi_{x_i}  (D^2 \Phi)^{-\frac{1}{2}} \bigr\|^2_{\rm HS}  \   d\mu.
\end{equation}

2) Triangular mappings.

Mappings of this type have the form
$$
T =  (T_1(x_1), T_2(x_1,x_2), \cdots, T_d(x_1, \cdots, x_d) ),
$$
where every $T_i$ is increasing in $x_i$.

It is easy to check
that in this case
\begin{equation}
\label{Tr1}
\int  V^2_{x_i} \ d \mu
=
\int \langle D^2 W(T) \cdot \partial_{x_i} T, \partial_{x_i} T \rangle \ d\mu +
\sum_{k=i}^{d} \int \Bigl( \frac{\partial_{x_i x_k} T_k}{\partial_{x_k} T_k}
 \Bigr)^2 \ d \mu,
 \end{equation}
\begin{equation}
\label{Tr2}
\mathcal{I}_{\mu}
=
\int \mbox{\rm Tr} \Bigl[ DT \cdot D^2W(T)\cdot (DT)^* \Bigr] \ d \mu
+
\sum_{k=1}^{d} \int |\nabla \ln \partial_{x_k} T_k |^2 \ d \mu.
\end{equation}

\section{Main result}

Recall that a function  $W$ is called uniformly convex (uniformly $K$-convex)  if
\begin{equation}
\label{-K-conv}
x \mapsto W(x) - \frac{K}{2} x^2
\end{equation}
is a convex function for $K \ge 0$.
For a smooth $W$ this is equivalent to the condition $D^2 W \ge K \cdot \mbox{Id}$.
Everywhere in this paper we deal with the case $K>0$ only.

One can introduce in the standard way the weighted Sobolev spaces $W^{2,p}(\mu)$.
We say that $f  \in L^2(\mu)$ admits a distributional derivative $f_{x_i} \in L^1(\mu)$ if
$$
\int f_{x_i} \xi \ d \mu = - \int f \xi_{x_i} \ d \mu + \int f V_{x_i} \xi \ d \mu
$$
for every test function $\xi$.
Similarly one can define $W^{2,p}_0(\mu)$ as a completion of the test functions in the corresponding Sobolev norm.
It is known that $W^{2,p}(\mu) = W^{2,p}_0(\mu)$ if $\mathcal{I}_{\mu} < \infty$ (see Theorem 5.1 in \cite{Eb}).

We denote by $f^{+}$ the function $\max\{f,0\}$ and by $A^+$
the positive part of a symmetric matrix $A$ (or zero matrix if $A \le 0$).

\begin{theorem}
\label{mr}
Assume that $\mathcal{I}_{\mu} < \infty$, $\mu$  admits the finite second moment,
and
$W$ satisfies (\ref{-K-conv}) for some $K>0$.
Then $\Phi \in W^{2,2}(\mu)$ and
\begin{equation}
\label{transport}
\mathcal{I}_{\mu}   \ge K
\int \| D^2 \Phi\|^2_{\rm HS} \ d\mu.
\end{equation}
\end{theorem}
\begin{proof}
{\bf Step 1} ($V$ and $W$ are smooth). Assume, in addition, that $V$ and $W$ satisfy the following assumptions
\begin{itemize}
\item[1)]
$V, W \in C^{\infty}(\R^d$)  and bounded from below
\item[2)]
$ D^2 V \le c \cdot \mbox{\rm Id}$ for some $c \in \R$.
\end{itemize}
By the Caffarelli's regularity results (see, for instance, Theorem 4.14 of \cite{Vill} and some justification in
\cite{Kol}, Section 4) $\Phi$ is smooth.
Moreover, it follows by the Caffarelli-type arguments from 2) and the uniform convexity
 of $W$ that  $$\sup_{x \in \R^d} \| D^2 \Phi(x)\| < C$$  for some $C$ (see, for instance, Theorem 2.2 in \cite{Kol}
and an independent proof in Section 6 below).

Let us show that (\ref{Fisher-i}) holds.
We take a smooth compactly supported test function $\xi$.
Multiply (\ref{vee}) by $\xi$ and integrate over $\mu$. Apply integration-by-parts formula (see (\ref{MAIBP})).
One obtains
\begin{align}
\label{xi-main}
\int   & V^2_{x_i} \  \xi  \  d \mu
\\&
\nonumber
=
\int \langle D^2 \Phi \cdot D^2 W(\nabla \Phi) \cdot D^2 \Phi \cdot e_i, e_i \rangle  \xi \ d \mu
+
\int \bigl\| (D^2 \Phi)^{-\frac{1}{2}} D^2 \Phi_{x_i}  (D^2 \Phi)^{-\frac{1}{2}} \bigr\|^2_{\rm HS}  \xi \   d\mu
\\&
\nonumber
+ \int \partial_{e_i} \xi \cdot V_{x_i} \ d \mu + \int \langle \nabla \xi,  (D^2 \Phi)^{-1} D^2 \Phi_{x_i}  \cdot e_i \rangle \ d\mu.
\end{align}
Assume that $\xi$ has the form $\xi = \eta(\nabla \Phi)$, where $\eta$ is a test function. One has $\nabla \xi = D^2 \Phi \cdot \nabla \eta(\Phi)$.
 Using the uniform estimate of $\| D^2 \Phi\|$ one obtains
$$
\bigl| \int \partial_{e_i} \xi \cdot V_{x_i} \ d \mu \bigr|
\le C \int |\nabla \eta(\nabla \Phi)| |V_{x_i}|\ d\mu \le
C \mathcal{I}^{\frac{1}{2}}_{\mu} \bigl( \int |\nabla \eta|^2 \ d\nu\bigr)^{\frac{1}{2}}.
$$
To estimate the last term we integrate by parts
\begin{align*}
 & \int \langle \nabla \xi,  (D^2 \Phi)^{-1} D^2 \Phi_{x_i}  \cdot e_i \rangle \ d\mu
=
\int \langle \nabla \eta(\nabla \Phi),   D^2 \Phi_{x_i}  \cdot e_i \rangle \ d\mu
\\&
= -
\int \langle  D^2 \eta(\nabla \Phi)  D^2 \Phi \cdot e_i,   D^2 \Phi   \cdot e_i \rangle \ d\mu
+
\int \langle \nabla \eta(\nabla \Phi),   D^2 \Phi  \cdot e_i \rangle V_{x_i} \ d\mu.
\end{align*}
The latter does not exceed
\begin{align*}&
C^2 \int \|  D^2 \eta(\nabla \Phi) \| \ d\mu
+ C
\int | \nabla \eta(\nabla \Phi)|  |V_{x_i}| \ d\mu
\\& \le
C^2 \int \| D^2 \eta \|  d \nu
+
C \mathcal{I}^{\frac{1}{2}}_{\mu} \bigl( \int |\nabla \eta|^2 \ d\nu\bigr)^{\frac{1}{2}}.
\end{align*}
Choosing a sequence of test function $\{\eta_n\}$ such that $0 \le \eta_n \le 1$, $\eta_n \to 1$ uniformly on every compact set, and $|\nabla \eta_n|^2 \to 0,\| D^2 \eta_n\|  \to 0$ in $L^1(\nu)$,
we get (\ref{Fisher-i}) (hence (\ref{transport})) for $V, W$ satisfying 1)-2).

{\bf Step 2} ($W$ is smooth).  Fix a  smooth uniform $K$-convex function $W$  and approximate $\mu$ by smooth measures.  We choose a sequence of functions $\{V_n\}$ such that every $V_n$ satisfies 1)-2) . In addition, we assume that
$\sqrt{\rho_n} \to \sqrt{\rho}$ in $W^{1,2}(\R^d)$, every $\mu_n = \rho_n \ dx = e^{-V_n} \ dx$ is a  probability measure,
and $\sup_n \int |x|^2 d \mu_n < \infty$.

Note that there exists a subsequence of $\{\nabla \Phi_n\}$ (denoted again by $\{\nabla \Phi_n\}$)
such that $\nabla \Phi_n \to \nabla \Phi$ almost everywhere.
Indeed, let $\Psi_n$ be the convex conjugated function to $\Phi_n$. Remind that $\nabla \Phi_n$ and $\nabla \Psi_n$
are reciprocal. One has $\sup_n \int |\nabla \Psi_n|^2 \  d \nu = \sup_n \int |x|^2 \ d \mu_n < \infty.$
We also require without loss of generality that $\int \Psi_n \ d \nu = 0$ (note that $\Psi_n \in L^2(\nu)$ by the Poncar{\'e} inequality
for uniform log-concave measures:  $K \int \bigl(f - \int f \ d \nu \bigr)^2 \ d\nu \le \int |\nabla f |^2 \ d \nu$).
Since $W$ is smooth, $\sup_n \int_{B_r} |\nabla \Psi_n|^2 \ d x < \infty$ for every ball $B_r$.
Using compactness of  Sobolev embeddings one can easily show that there exists
an a.e. convergent subsequence (denoted again by $\{\Psi_n\}$ ) $\Psi_n \to \Psi$. Since $\Psi_n$ are convex, one also has $\nabla \Psi_n \to \nabla \Psi$
a.e. This implies a.e. convergence of the convex conjugated potentials $\Phi_n \to \Phi$ and their gradients
 $\nabla \Phi_n \to \nabla \Phi$ .

Moreover, since
$$
\int \|\nabla \Phi_n\|^2 \ d\mu_n = \int \|x\|^2 \ d\nu = \int \|\nabla \Phi\|^2 \ d\mu,
$$
one has $\nabla \Phi_n \cdot \sqrt{\rho_n} \to \nabla \Phi \cdot \sqrt{\rho}$ strongly in $L^2(\R^d)$.
In the same way one can check that  (again up to a subsequence) $\partial_{x_i x_j} \Phi_n \sqrt{\rho_n}$ converges
weakly in $L^2(\R^d)$ to some function $F$. This implies
$$
\int \xi  \cdot \partial_{x_i x_j} \Phi_n \sqrt{\rho_n} \ \sqrt{\rho_n} \ dx
\to
\int \xi \cdot  F \ \sqrt{\rho} \ dx.
$$
In the other hand
$$
\int \xi \cdot \partial_{x_i x_j}  \Phi_n \ \rho_n \  dx
=
-
\int \xi_{x_j} \cdot \partial_{x_i} \Phi_n \ \rho_n \ dx
-
\int \xi \cdot \partial_{x_i} \Phi_n \ \frac{\partial_{x_j} \rho_n}{\sqrt{\rho_n}} \cdot \sqrt{\rho_n} \  dx.
$$
By the strong convergence $\nabla \Phi_n \sqrt{\rho_n} \to \nabla \Phi \sqrt{\rho} $ the latter tends to
$$
-
\int \xi_{x_j} \cdot \partial_{x_i} \Phi \ \rho \ dx
-
\int \xi \cdot \partial_{x_i} \Phi \ \partial_{x_j}  \rho\  dx.
$$
The relation
$$
\int \xi \cdot  F \ \sqrt{\rho} \ dx
=
-
\int \xi_{x_j} \cdot \partial_{x_i} \Phi \ \rho \ dx
-
\int \xi \cdot \partial_{x_i} \Phi \ \partial_{x_j} \rho  \  dx
$$
 implies that the second distributional derivative $\partial_{x_i x_j}  \Phi$ equals to $F/\sqrt{\rho}$. Hence
$ D^2 \Phi_n \cdot \sqrt{\rho_n} \to D^2 \Phi \cdot \sqrt{\rho}$ weakly in $L^2(\R^d)$.
Since the statement holds for the approximating sequence (according to Step 1), by the
standard property of the weak convergence
$$
\mathcal{I}_{\mu}= \lim_n \mathcal{I}_{\mu_n} \ge \underline{\lim}_n \int \|D^2 \Phi_n\|^2 \ d\mu_n \ge \int \|D^2 \Phi\|^2 \ d\mu.
$$

{\bf Step 3}. At the final step we fix $\mu$ and approximate $e^{-W}$ by smooth uniformly log-concave probability densities $e^{-W_n}$ such that
$\int |x|^2 \ d\nu_n \to \int |x|^2 \ d\nu$  and (\ref{-K-conv}) holds for every $W_n$.
The proof follows the arguments of Step 2. It is  even easier because one has to deal with the fixed reference measure $\mu$.
One obtains that $\nabla \Phi_n \to \nabla \Phi$ strongly in $L^2(\mu)$ and $D^2 \Phi_n \to D^2 \Phi$ weakly in $L^2(\mu)$.
The result follows from the standard properties of the weak convergence.
\end{proof}

\begin{remark} {\bf Third-order derivatives.}
Note that  some global bounds on the {\it third} derivatives of $\Phi$
are also available. Indeed, if $\Phi$ is sufficiently smooth and (\ref{Fisher-i}) holds, then
\begin{align*}
\int V^2_{x_i} \ d\mu  & \ge K \int \|D^2 \Phi \cdot e_i \|^2 \ d\mu
+
\int \bigl\| (D^2 \Phi)^{-\frac{1}{2}} D^2 \Phi_{x_i}  (D^2 \Phi)^{-\frac{1}{2}} \bigr\|^2_{\rm HS} \ d\mu
\\&
\ge
 K \int \|D^2 \Phi \cdot e_i\|^2 \ d\mu
+ \int \frac{\bigl\| D^2 \Phi_{x_i} \bigr\|^2_{\rm HS}}{\|D^2 \Phi\|^2}   \ d\mu,
\end{align*}
where $\| \cdot \|$ is the standard operator norm. 
Summing over $i$, bounding the operator norm by the Hillbert-Schmidt norm, and applying
the Cauchy inequality, one obtains
$$ \int |\nabla V|^2 \ d\mu
\ge 2 \sqrt{K} \int  \Bigl[ \sum_{i=1}^d  \| D^2 \Phi_{x_i}\|^2_{\rm HS}  \Bigr]^{\frac{1}{2}} \ d \mu.
$$
\end{remark}

\section{Transportation inequalities}

In this section we show that inequality (\ref{main}) follows from a (generalized) Talagrand inequality.

The following generalization of the Talagrand inequality has been proved in \cite{Kol04}.
Let $f \cdot \nu$, $g \cdot \nu$ be  probability measures, $\nu = e^{-W} \ dx$ with
$D^2W \ge K \cdot {\mbox Id}$, $K >0$. Let $T_f$ ($T_g$ ) be the optimal transportation mapping
pushing forward $f \cdot \nu$ ($g \cdot \nu$) onto $\nu$. Then the following inequality holds
\begin{equation}
\label{fg}
\int f \log \frac{f}{g} \ d \nu \ge \frac{K}{2} \int |T_f - T_g|^2 f \ d\nu.
\end{equation}

\begin{remark}
The Talagrand inequality in its classical form
\begin{equation}
\label{fg0}
\int \rho \log \rho \ d\nu \ge \frac{K}{2} \int |T(x) - x|^2 \rho \ d\nu
\end{equation}
holds for any reasonable transportation mapping $T$ sending $\rho \cdot \nu$ onto $\nu$
and satisfying
\begin{equation}
\label{log-det}
\mbox{div} (T^{-1}) -d - \log \det D(T^{-1}) \ge 0
\end{equation}
(this can be checked by the standard transportational arguments, see, for instance, \cite{Led}).
 Then (\ref{fg}) follows from
(\ref{fg0}) if we set
$$
\rho = \frac{f}{g} \circ (T_g^{-1}), \ \ T = T_f \circ T^{-1}_g.
$$
Note  that  (\ref{log-det}) holds for $T$ because $D(T^{-1})$ is a composition of two non-negative matrix (see arguments below in
the proof of  Theorem \ref{increment}).
\end{remark}

Let us apply (\ref{fg}) to $f(x) = e^{-V(x) +W(x)}$ and $g(x) = e^{-V(x+e) + W(x)}$ ($e$ is a fixed vector).
Clearly, $T_f = \nabla \Phi$ is the optimal transportation between $\mu$ and $\nu$ and $T_g = \nabla \Phi(x+e)$.
We obtain
$$
\int (V(x+e) - V(x)) \ d\mu \ge \frac{K}{2} \int |\nabla \Phi(x+e) - \nabla \Phi(x)|^2 d\mu.
$$

In order to make the paper self-contained, we give below an independent prove of this result. Then we deduce from it the main result of the paper (inequality (\ref{main})).

Recall that every convex function $\varphi$ admits a.e. the so-called Alexandrov second-order  derivative $D^2_a \varphi$,
which is the absolutely continuous part of its distributional derivative $D^2 \varphi$.

The following lemma holds trivially for smooth mappings and can be easily checked  by approximation arguments.

\begin{lemma}
\label{compos}
Let $\varphi: A \to \R$, $\psi: B \to \R$ be convex functions on convex sets $A$, $B$. Assume that $\nabla \psi(B) \subset A$.
Then
$$
\mbox{\rm div} (\nabla \varphi \circ \nabla \psi)
\ge \mbox{\rm Tr}\bigl[ D^2_a \varphi (\nabla \psi) \cdot D^2_a \psi \bigr] \ dx \ge 0,
$$
where $\mbox{\rm div}$ is the distributional derivative.
\end{lemma}

\begin{theorem}
\label{increment}
Assume that $W$ is $K$-uniformly convex. Then for every $e \in \R^d$
$$
\int (V(x+ e) - V(x))  \ d\mu
\ge
\frac{K}{2} \int  |\nabla \Phi(x+e)- \nabla \Phi(x)|^2 \ d\mu.
$$
\end{theorem}
\begin{proof}
By a result of R.J.~McCann on the change of variables formula (see  \cite{McCann2} or \cite{Vill})
$$
e^{-V} = {\det}_a D^2 \Phi \cdot e^{-W(\nabla \Phi)}
$$
$\mu$-a.e.
Hence
$V = W(\nabla \Phi) - \log {\det}_a D^2 \Phi$ and
\begin{align*}
V(x+ e) - V(x)  & =
 W(\nabla \Phi(x+e)) - W(\nabla \Phi(x))
\\& - \log \Bigl[ ({\det}_a D^2 \Phi(x))^{-1}  \cdot {\det}_a D^2\Phi(x+e)\Bigr].
\end{align*}
By the $K$-uniform convexity of $W$
\begin{align*}
 W(\nabla \Phi(x+e)) - W(\nabla \Phi(x))  &
\ge
\langle \nabla \Phi(x+e)- \nabla \Phi(x), \nabla W(\nabla \Phi(x))  \rangle
\\& + \frac{K}{2} |\nabla \Phi(x+e)- \nabla \Phi(x)|^2.
\end{align*}
This implies
\begin{align*}
\int (V(x+ e) - V(x))  \ d\mu &
\ge
\int \frac{K}{2} |\nabla \Phi(x+e)- \nabla \Phi(x)|^2 \ d\mu
\\&
+ \int  \langle \nabla \Phi(x+e)- \nabla \Phi(x), \nabla W(\nabla \Phi(x))  \rangle  \ d\mu
\\& - \int \log \Bigl[ ({\det}_a D^2 \Phi(x))^{-1}  \cdot {\det}_a D^2\Phi(x+e)\Bigr] \ d \mu.
\end{align*}
Denote by $\Psi=\Phi^*$ the convex conjugated function of $\Phi$.
Using the fact that $\nabla \Psi$ and $\nabla \Phi$ are reciprocal we get
\begin{align*}
 & \int  \langle \nabla \Phi(x+e)- \nabla \Phi(x), \nabla W(\nabla \Phi(x))  \rangle  \ d\mu
= \int  \langle \nabla \Phi(\nabla \Psi(x)+e)- x, \nabla W(x)  \rangle  \ d\nu
\\&
=
\int \mbox{\rm{div}} \bigl(  \nabla \Phi(\nabla \Psi(x)+e)- x\bigr)  e^{-W},
\end{align*}
where $ \mbox{\rm{div}} \bigl(  \nabla \Phi(\nabla \Psi(x)+e)- x\bigr) $
is the distributional derivative of the vector field $\nabla \Phi(\nabla \Psi(x)+e)- x$.

By Lemma \ref{compos} and the relation $ (D^2_a \Phi(\nabla \Psi) )^{-1} = D^2_a \Psi$ which holds $\nu$-a.e. 
(see \cite{McCann2} or \cite{Vill}), we get
\begin{align*}
\int \mbox{\rm{div}} \bigl(  \nabla \Phi(\nabla \Psi(x)+e)- x\bigr)  e^{-W} &
\ge
\int \bigl( \mbox{\rm Tr} D^2_a \Phi (\nabla \Psi(x)+e)  \cdot  D^2_a \Psi(x)-  d \bigr) \ d\nu
\\&
=\int \bigl( \mbox{\rm Tr} D^2_a \Phi(x+ e)   \cdot  (D^2_a \Phi(x))^{-1}-  d \bigr) \ d \mu.
\end{align*}
It remains to note that
$$
  \mbox{\rm Tr} D^2_a \Phi(x+ e)   \cdot  (D^2_a \Phi(x))^{-1}-  d
-  \log \Bigl[ {\det}_a D^2 \Phi(x)  \cdot ({\det}_a D^2\Phi(x+e))^{-1}\Bigr]
\ge 0.
$$
Indeed, if $A$ and $B$ are symmetric and non-negative, then
$$
\mbox{Tr} AB - d - \log \det AB = \mbox{Tr} C - d - \log \det C,
$$
where $C = B^{1/2} A B^{1/2}$ is a symmetric non-negative matrix.
It is well-known that $\mbox{Tr} C - d - \log \det C >0$. Indeed, the latter is equal to $\sum_{i} (c_i -1 - \log c_i) \ge 0$, where $c_i$ are eigenvalues of $C$. The proof is complete.
\end{proof}

\begin{proposition}
Inequality (\ref{inc-ineq}) implies (\ref{main}).
\end{proposition}
\begin{proof}
Following the arguments of Theorem \ref{mr} we see that it is sufficient to establish implication
(\ref{inc-ineq})  $\Rightarrow$ (\ref{main}) for a nice potential $V$.
By Theorem \ref{increment}
\begin{align*}
 \int & \frac{V(x+te)  + V(x-te) - 2 V(x)}{t^2} \ d \mu  \\&
\ge \frac{K}{2t^2} \int \Bigl( \| \nabla \Phi(x+te) - \nabla \Phi(x)\|^2 + \| \nabla \Phi(x-te) - \nabla \Phi(x)\|^2\Bigr) \ d\mu.
\end{align*}
Thus, without loss of generality we may assume that $V$ satisfies
$$
\frac{\exp( V(x) - V(x+te)) -1 }{t} \to V_e \  \mbox{in} \ L^2(\mu), \  t \to 0
$$
and
$$
 {\lim}_{t \to 0} \frac{1}{t^2}\int (V(x+te) + V(x-te) - 2 V(x)) \ d \mu = \int V_{ee} \ d\mu
$$
for every $e$.
Extract $L^2(\mu)$-weakly convergent subsequences $\bigl\{ \frac{\nabla \Phi(x \pm t_n e) - \nabla \Phi(x)}{t_n} \bigr\}$ 
(we keep the same index $n$).
Note that
 \begin{align*}
& \int \frac{\nabla \Phi(x + t_n e) - \nabla \Phi(x)}{t_n}  \  \xi \ d\mu
= \int \nabla \Phi(x) \frac{ \xi(x-t_ne) - \xi(x)}{t_n} d\mu
\\& +
\int \nabla \Phi(x) \ \xi  \ \frac{\exp(V(x) - V(x-t_ne)) -1 }{t_n}  \  d\mu.
\end{align*}
Obviously, the latter tends to
$$
- \int \nabla \Phi(x) \xi_{e} \ d\mu
+
\int \nabla \Phi(x) \ \xi  \ V_e \  d\mu.
$$
Hence
$$
\frac{\nabla \Phi(x \pm  t_n e) - \nabla \Phi(x)}{t_n}  \to  \nabla \Phi_e
$$
weakly in $L^2(\mu)$.
By  the properties of the weak convergence.
$$
\int  V^2_e d\mu  = \int V_{ee} \ d\mu
\ge K \int \| \nabla \Phi_e \|^2   \  d\mu.
$$
Applying this to every $e_i$ and taking the sum we complete the proof.
\end{proof}

\section{Dimension-free inequalities}

In this section we prove some essentially infinite-dimensional estimates (which do not contain dimension-dependent constants and
make sense in the infinite-dimensional case).
The results below also hold  (with certain modifications) for the triangular mappings.

\subsection{Gaussian case}
We denote by $\gamma$ the standard Gaussian measure on $\R^d$.
Let $\mu = g \cdot \gamma$, $\nu = \gamma$ and $\nabla \Phi$ be the corresponding optimal transport.
According to the result from Section 3
$$
\int \Bigr| \frac{\nabla g}{g} - x \Bigr|^2 g d \gamma
=
 \int \|D^2 \Phi \|^2_{HS}  g d \gamma
+
\sum_{k=1}^{d} \int \mbox{Tr} \bigl[ (D^2 \Phi)^{-1} D^2 \Phi_{x_k} \bigr]^2 \ g d \gamma.
$$

Note that
$$
\int \Bigr| \frac{\nabla g}{g} - x \Bigr|^2 \ g d \gamma
=
\int \frac{|\nabla g|^2}{g} d \gamma
- 2 \int \langle \nabla g, x\rangle d \gamma
+ \int |x|^2 \ g d \gamma
$$
and
$$
 \int \|D^2 \Phi \|^2_{HS}  g d \gamma
 =
 \int \|D^2 \Phi - \mbox{Id} \|^2_{HS} \  g d \gamma
 + 2 \int \Delta \Phi\ g  d \gamma -d.
$$
Apply integration-by-parts
$$
- 2 \int \langle \nabla g, x\rangle  \ d \gamma
+ \int |x|^2 \ g d \gamma
=
2d - \int |x|^2 \ g d \gamma.
$$
By the change of variables formula
$$
2 \int \Delta \Phi \  g  d \gamma -d =
2 \int \Delta \Phi \ g  d \gamma - \int |\nabla \Phi|^2 \  g d \gamma.
$$
Consequently
\begin{align*}
\int \frac{|\nabla g|^2}{g} d \gamma
&
=
 \int \|D^2 \Phi -\mbox{Id} \|^2_{HS}  g d \gamma
+
 \int ( |x|^2 - |\nabla \Phi|^2) \ g d \gamma
 + 2 \int (\Delta \Phi  - d) g  d \gamma
\\& +
\sum_{k=1}^{d} \int \mbox{Tr} \bigl[ (D^2 \Phi)^{-1} D^2 \Phi_{x_k} \bigr]^2 \ g d \gamma.
\end{align*}
Taking the logarithm of the change of variables formula we get
$$
\log g = \frac{|x|^2}{2} - \frac{|\nabla \Phi|^2}{2} + \log \det D^2 \Phi.
$$
Applying this formula we get the heuristic proof of the following statement:

{\it
Every probability measure $g \cdot \gamma$ with smooth $g$ and smooth $\nabla \Phi$ satisfies the following relation
\begin{align}
\label{infdim}
 \mbox{\rm I}_{\gamma} g  & =   2 \mbox{\rm Ent}_{\gamma} g
- 2 \int \log {\det}_2 (D^2 \Phi - \mbox{\rm Id})\ g d \gamma
\\& \nonumber
+
 \int \|D^2 \Phi -\mbox{\rm Id} \|^2_{HS} \  g d \gamma
 +
\sum_{k=1}^{d} \int \mbox{\rm Tr} \bigl[ (D^2 \Phi)^{-1} D^2 \Phi_{x_k} \bigr]^2 \ g d \gamma,
\end{align}
where
$ \mbox{\rm I}_{\gamma} g =
\int \frac{|\nabla g|^2}{g} d \gamma$ (relative information),
$\mbox{\rm Ent}_{\gamma} g  = \int g \log g \ d\gamma$ (relative entropy),
${\det}_2 (D^2 \Phi - \mbox{\rm Id}) = \det D^2 \Phi \cdot \exp \bigl(d - \Delta \Phi \bigr)$
(the Fredholm-Carleman determinant of  $D^2 \Phi - \mbox{\rm Id}$).
}

\begin{remark}
Since all the terms in the right-hand side are non-negative,  this statement  implies, in particular, the classical logarithmic Sobolev inequality
$$
 \mbox{\rm I}_{\gamma} g   \ge   2 \mbox{\rm Ent}_{\gamma} g
$$
and the Gaussian analog of (\ref{main})
\begin{equation}
\label{Tal2}
\mbox{\rm I}_{\gamma} g  \ge
 \int \|D^2 \Phi -\mbox{\rm Id} \|^2_{HS} \  g d \gamma.
\end{equation}
\end{remark}

\begin{remark}
Identity (\ref{infdim}) holds, for instance, under assumptions: $g$ is smooth, bounded, stricktly positive, $\mbox{\rm I}_{\gamma} g   < \infty$, and
$-D^2 \log g \le c \cdot \mbox{Id}$. See  Step 1 in the proof of Theorem \ref{mr}.

Inequality (\ref{Tal2}) follows immediately from Theorem \ref{mr} under the unique assumption ${\mbox \rm I}_{\gamma}g < \infty$.
\end{remark}

\begin{remark}
\label{extrem}
It was pointed out to the author by Michel Ledoux that (\ref{infdim}) implies the description of the extremals for the classical
log-Sobolev inequality. 
Indeed, the case of equality in (\ref{infdim}) is possible if and only if $D^2 \Phi = \mbox{\rm Id}$, hence $\nabla \Phi$ is linear and $g$
has the form $g = \exp(\langle h, x \rangle - \frac{1}{2} \| h\|^2)$, $h \in \R^d$. 
This result has been established by other methods in \cite{Carlen}.
\end{remark}

\subsection{Log-concave case}

Below we deal with the case $\mu = g e^{-W} \ dx$, $\nu=e^{-W} dx$, where $W$ is convex.
By the above results
$$
\int \Bigl|\frac{\nabla g}{g} -   \nabla W \Bigr|^2 \  g d\mu
\ge
  \int \mbox{Tr} \Bigl[ D^2 \Phi  \cdot D^2W(S)  \cdot D^2 \Phi \Bigr] \ g d \mu.
$$
Rewrite the left-hand side
$$
\int \Bigl|\frac{\nabla g}{g} -   \nabla W \Bigr|^2 \ g d  \mu
=
\int \frac{|\nabla g|^2}{g}\ d\mu
- 2 \int \langle \nabla g, \nabla W \rangle \ d \mu
+ \int \bigl|\nabla W \bigr|^2 \ g d \mu.
$$
Rewrite the right-hand side
\begin{align*}
 \int \mbox{Tr} & \Bigl[ D^2 \Phi  \cdot D^2W(\nabla \Phi)  \cdot D^2 \Phi  \Bigr]  \ g   d \mu
 =
 \int \mbox{Tr} \Bigl[ (D^2 \Phi - \mbox{Id})  \cdot D^2W(\nabla \Phi)  \cdot (D^2 \Phi  -\mbox{Id}) \Bigr] \ g   d \mu
\\& + 2 \int \mbox{div} (\nabla W \circ \nabla \Phi) \ g d \mu
-
 \int \Delta W (\nabla \Phi) \  g d\mu
\\& =
 \int \mbox{Tr} \Bigl[ (D^2 \Phi - \mbox{Id})  \cdot D^2W(\nabla \Phi)  \cdot (D^2 \Phi  -\mbox{Id}) \Bigr]  \ g  d \mu
\\& - 2 \int \langle \nabla g, \nabla W \circ \nabla \Phi \rangle d \mu
+ 2 \int \langle \nabla W, \nabla W \circ \nabla \Phi \rangle \ g \ d \mu
-
 \int \Delta W (\nabla \Phi) \  g d\mu .
\end{align*}
Consequently
\begin{align*}
\int \frac{|\nabla g|^2}{g}\ d\mu &
+ \int \bigl|\nabla W \bigr|^2 \ g d \mu
\\&
\ge
2 \int \langle \nabla g, \nabla W  - \nabla W \circ \nabla \Phi\rangle \ d \mu
+ 2 \int \langle \nabla W, \nabla W \circ \nabla \Phi\rangle \ g \ d \mu
\\&
+  \int \mbox{Tr} \Bigl[ (D^2 \Phi- \mbox{Id})  \cdot D^2W(\nabla \Phi)  \cdot (D^2 \Phi -\mbox{Id}) \Bigr] \  g   d \mu
-
 \int \Delta W (\nabla \Phi) \  g d\mu .
\end{align*}
This implies
\begin{align*}
\int & \frac{|\nabla g|^2}{g}\ d \mu
+ \int \bigl|\nabla W \bigr|^2 \ g d \mu
- 2 \int \langle \nabla W, \nabla W \circ \nabla \Phi\rangle \ g \ d \mu
+ \int |\nabla W \circ \nabla \Phi|^2 \ g d \mu \ge
\\&
\ge
2 \int \langle \nabla g, \nabla W  - \nabla W \circ \nabla \Phi\rangle \ d \mu
\\&
 +  \int \mbox{Tr} \Bigl[ (D^2 \Phi- \mbox{Id})  \cdot D^2W(\nabla \Phi)  \cdot (D^2 \Phi -\mbox{Id}) \Bigr] \ g  d \mu
+
 \int \bigl[ |\nabla W \circ \nabla \Phi|^2 - \Delta W (\nabla \Phi) \bigr] \  g d\mu .
\end{align*}
Taking into account that
$$
 \int \bigl[ |\nabla W \circ \nabla \Phi|^2 - \Delta W (\nabla \Phi) \bigr] \  g d\mu
=
 \int \bigl[ |\nabla W |^2 - \Delta W \bigr] \   d\mu  = 0
$$
we get
\begin{equation}
\int \Bigl|  \frac{\nabla g}{g} - \bigl(  \nabla W  - \nabla W \circ \nabla \Phi \bigr) \Bigr|^2\  g d \mu
\ge  \int \mbox{Tr} \Bigl[ (D^2 \Phi- \mbox{Id})  \cdot D^2W(\nabla \Phi)  \cdot (D^2 \Phi -\mbox{Id}) \Bigr] \ g  d \mu.
\end{equation}

By the Cauchy inequality
\begin{equation}
\label{non-gauss}
2 \int \frac{|\nabla g|^2}{g}\ d \mu
+ 2 \int \bigl|\nabla W  - \nabla W \circ \nabla \Phi\bigr|^2  \ g  d \mu
\ge  \int \mbox{Tr} \Bigl[ (D^2 \Phi- \mbox{Id})  \cdot D^2W(\nabla \Phi)  \cdot (D^2 \Phi -\mbox{Id}) \Bigr] \ g  d \mu.
\end{equation}

Thus in order to estimate
$$\int \mbox{Tr} \Bigl[ (D^2\Phi- \mbox{Id})  \cdot D^2W(\nabla \Phi)  \cdot (D^2\Phi -\mbox{Id}) \Bigr] \ g  d \mu$$
(or $\int \|D^2\Phi - \mbox{Id}\|^2_{HS} \  g  d \mu$ for uniformly convex $W$) it is sufficient to get a bound for
$$
\int \bigl|\nabla W  - \nabla W \circ \nabla \Phi \bigr|^2  \ g  d \mu.
$$
Some estimates of quantities of this type are established in \cite{BoKo}. We give below the proof for the most simple case (the potential has a quadratic-like growth).

\begin{theorem}
Assume that for some $K>0$
$$
W(x) - \langle \nabla W(y),x - y \rangle - W(y) \ge \frac{K}{2} |\nabla W(x) - \nabla W(y)|^2
$$
and
$
D^2 W \ge K \cdot \mbox{\rm Id}.
$
Then
$$
\frac{K}{2} \int \|D^2 \Phi - \mbox{\rm Id}\|^2_{HS} \ g  d \mu
\le
\frac{2}{K} \int g \log g \ d\mu + \int \frac{|\nabla g|^2}{g} \ d \mu.
$$
In particular, the estimate holds for some $K>0$ if $C_1 \cdot \mbox{\rm Id} \le D^2 W \le C_2  \cdot \mbox{\rm Id}.$
\end{theorem}
\begin{proof}
The result follows from Theorem \ref{mr}, the above computations, and the estimate below. The proof of the result can be easily reduced to the case
of smooth $g$ and $T$ (see the proof of Theorem \ref{mr}). By the change of variables formula for $T = (\nabla \Phi)^{-1} = \nabla \Phi^*$ one has
$$
\log g(T) - W(T) + \log \det DT = - W(x) .
$$
Rewrite it in the following way
$$
\log g(T) = W(T) - \langle \nabla W(x),T(x) - x \rangle - W(x)  + \Bigl[ \langle \nabla W(x),T(x) - x \rangle -  \log \det DT \Bigr].
$$
Note that
$$
\int \Bigl[ \langle \nabla W(x),T(x) - x \rangle -  \log \det DT \Bigr] \ d\mu
=
\Bigl[ \mbox{Tr} DT - d -  \log \det DT \Bigr] \ d\mu \ge 0.
$$
Hence
$$
\int \log g(T) \ d \mu \ge  \int \Bigl[W(T) - \langle \nabla W(x),T(x) - x \rangle - W(x) \Bigr] \ d\mu .
$$
By the change of variables
$$
\int g  \log g \ d \mu \ge  \int \Bigl[ W(x) - \langle \nabla W(\nabla \Phi(x)) ,x - \nabla \Phi \rangle - W(\nabla \Phi) \Bigr] \ g(x) d\mu .
$$
This inequality, (\ref{non-gauss}), and the assumptions of the Theorem imply the result.
\end{proof}

\section{$L^p$-estimates and the Caffarelli's theorem}

We generalize below the results of the previous sections and prove some corresponding $L^p$-estimates.
As a particular case we get the contraction result of  Caffarelli. Note that some dimension-free $L^p$-generalizations 
of the Talagrand transportation inequality have been obtained in \cite{BoKo}.
In particular, it was shown in \cite{BoKo} that $\|\nabla \Phi\|_{L^{2p}(\mu)}$ is controlled by
$\int g |\log g|^p \ d\mu$, $p \ge 1$ for any $\mu$ satisfying a log-Sobolev inequality.

The proof of the result below follows the arguments of Theorem \ref{increment}.
That is why we omit the details and just give a short outline of the proof.

\begin{theorem}
\label{lp-est}
 Assume that $D^2 W \ge K \cdot \mbox{\rm Id}$. Then for every unit $e$, $p\ge0$,  and $r= \frac{p+2}{2}$ one has
$$
K \| \Phi^2_{ee} \|_{L^{r}(\mu)} \le \| (V_{ee})_{+} \|_{L^{r}(\mu)},
$$
$$
K \| \Phi^2_{ee} \|_{L^{r}(\mu)} \le \frac{p+4}{4}  \| V^2_{e} \|_{L^{r}(\mu)}.
$$
\end{theorem}
\begin{proof}

 Fix unit vector $e$, apply the change of variables formula
and the uniform convexity of $W$
\begin{align*}
V(x+ te) - & V(x)  \ge 
\langle \nabla \Phi(x+te)- \nabla \Phi(x), \nabla W(\nabla \Phi(x))  \rangle \\& + \frac{K}{2} |\nabla \Phi(x+te)- \nabla \Phi(x)|^2
 - \log \Bigl[ ({\det}_a D^2 \Phi(x))^{-1}  \cdot {\det}_a D^2\Phi(x+te)\Bigr].
\end{align*}
Multiply this identity by $(\delta_{te} \Phi)^p $, where $p \ge 0$ and
$$
\delta_{te} \Phi =  \Phi(x+te) + \Phi(x-te) -2 \Phi(x)
$$
and integrate over $\mu$.
Integrating by parts
we get
\begin{align*}
& \int \langle   \nabla \Phi(x+te)- \nabla \Phi(x), \nabla W(\nabla \Phi(x))  \rangle 
(\delta_{te} \Phi)^p 
\ d\mu \\& =
\int \langle  \nabla \Phi(x+te) \circ (\nabla \Psi)- x, \nabla W(x)  \rangle 
(\delta_{te} \Phi)^p  \circ (\nabla \Psi)
\ d\nu
\\&
\ge
\int \Bigl( \mbox{Tr} \bigl[ D^2 _a\Phi(x+te) \cdot (D^2_a \Phi)^{-1} \bigr] \circ  (\nabla \Psi)  - d \Bigr) (\delta_{te} \Phi)^p  \circ (\nabla \Psi) \ d\nu
\\&
+ 
p \int \Big\langle \nabla \Phi(x+te) \circ (\nabla \Psi)- x, (D^2 \Psi)   \nabla \delta_{te} \Phi  \circ (\nabla \Psi)  \Big\rangle (\delta_{te} \Phi)^{p-1}\circ (\nabla \Psi)  \ d\nu.
\end{align*}
Applying the inequality $\mbox{Tr} A - d - \log \det A \ge 0$ which is valid for compositions of symmetric positive matrices we get
\begin{align*}
\int \bigl( V(x+ te) - & V(x) \bigr) (\delta_{te} \Phi)^p d\mu 
\ge
\frac{K}{2} \int |\nabla \Phi(x+te)- \nabla \Phi(x)|^2 (\delta_{te} \Phi)^p \ d\mu \\&
+ p \int \Big\langle \nabla \Phi(x+te) - \nabla \Phi(x), (D^2 \Psi) \circ \nabla \Phi(x)   \nabla \delta_{te} \Phi  \Big\rangle (\delta_{te} \Phi)^{p-1}  \ d\mu.
\end{align*}
Applying the same inequality to $-te$ and taking the sum we get
\begin{align*}
\int  & \bigl( V(x+ te) + V(x-te) - 2V(x) \bigr) (\delta_{te} \Phi)^p d\mu 
\\&
\ge
\frac{K}{2} \int |\nabla \Phi(x+te)- \nabla \Phi(x)|^2 (\delta_{te} \Phi)^p \ d\mu  + 
\frac{K}{2} \int |\nabla \Phi(x-te)- \nabla \Phi(x)|^2 (\delta_{te} \Phi)^p \ d\mu \\&
+ p \int \Big\langle \nabla \delta_{te} \Phi , (D^2_a \Phi)^{-1}    \nabla \delta_{te} \Phi  \Big\rangle (\delta_{te} \Phi)^{p-1}  \ d\mu.
\end{align*}
Note that the last term is non-negative.
Dividing by $t^{2p}$ and passing to the limit we obtain
\begin{equation}
\label{lp-sd+}
\int V_{ee}  \Phi_{ee}^p  \ d \mu 
\ge K
\int \| D^2 \Phi \cdot e\|^2   \Phi_{ee}^p  \ d \mu
+ p \int \langle (D^2 \Phi)^{-1}  \nabla \Phi_{ee}, \nabla \Phi_{ee} \rangle  \Phi_{ee}^{p-1} \ d \mu.
\end{equation}
For the proof of the first part we note that
$$
\int V_{ee} \Phi^p_{ee} \ d\mu 
\ge 
K \int  \Phi_{ee}^{p+2}  \ d \mu.
$$
Applying
the H{\"o}lder inequality
one gets
$$
 \| (V_{ee})_{+}\|_{L^{(p+2)/2}(\mu)}  \| \Phi_{ee}^p\|_{L^{(p+2)/p}(\mu)}
\ge \int V_{ee} \Phi^p_{ee} \ d\mu .
$$
This readily implies the result.

To prove the second part we integrate by parts the left-hand side
\begin{align*}
\int V_{ee} \Phi^{p}_{ee} \  d \mu &
=
-p \int V_e \Phi_{eee} \Phi^{p-1}_{ee} \ d\mu +  \int V^2_e \Phi^{p}_{ee} \  d \mu
\\&
= - p \int \langle \nabla \Phi_{ee}, V_e \cdot e \rangle \Phi^{p-1}_{ee} d\mu +  \int V^2_e \Phi^{p}_{ee} \  d \mu.
\end{align*}
By the Cauchy inequality the latter does not exceed
$$
 p \int \langle (D^2 \Phi)^{-1}  \nabla \Phi_{ee}, \nabla \Phi_{ee} \rangle  \Phi_{ee}^{p-1} \ d \mu
+ \frac{p}{4} \int V^2_e \langle (D^2 \Phi) e, e \rangle  \Phi_{ee}^{p-1} \ d \mu +  \int V^2_e \Phi^{p}_{ee} \  d \mu.
$$
Inequality (\ref{lp-sd+}) implies
$$
\frac{p+4}{4} \int V^2_e   \Phi_{ee}^{p} \ d \mu \ge 
 K \int |\nabla \Phi_e|^2 \Phi_{ee}^p  \ d \mu \ge K \int  \Phi_{ee}^{p+2}  \ d \mu.
$$
The rest of the proof is the same as in the first part.

\end{proof}

\begin{corollary}
In the limit  $p\to\infty$ we obtain  the contraction theorem of Caffarelli
$$
K \| \Phi_{ee} \|^2_{L^{\infty}(\mu)} \le \| (V_{ee})_{+} \|_{L^{\infty}(\mu)}.$$
\end{corollary}

\section{Operator norm estimates}

This section gives a partial answer to the question asked to the author by Emanuel Milman.
Is it possible to estimate effectively (say, without dimension dependence)
the operator norm of $D^2 \Phi$?
Estimates of this type would have interesting consequences for Sobolev-type
inequalities of log-concave measures.

Since the operator norm is
controlled by the Hilbert-Schmidt norm, the previous results imply trivally the following estimate
$$
\mathcal{I}_{\mu}  \ge K \int \|D^2 \Phi\|^2_{HS} \ d\mu \ge K \int \|D^2 \Phi\|^2 \ d\mu.
$$
We emphasize, however, that for many problems the assumption $\mathcal{I}_{\mu} < \infty$ is too strong and leads
to dimension dependent results.

The main aim of this section is to show that for the uniformly log-concave $\nu$
$$
\int \|(D^2 V)^{+}\| d\mu \ge K \int  \|D^2 \Phi\|^2 \ d\mu .
$$

\begin{lemma}
\label{eigenvec-ineq}
Assume that  $\Phi$ is smooth. Then for every smooth  vector field $v$ and  every nonnegative
test function $\eta$ the following inequality holds
\begin{align*}
\int \langle D^2 V v,v \rangle \eta & \  d \mu  
\ge K \int \| D^2 \Phi \cdot v \|^2  \eta  \  d \mu
+ \int \langle (D^2 \Phi)_v \cdot v  , (D^2 \Phi)^{-1}\nabla \eta \rangle  \ d\mu
\\&
+ 2  \int \mbox{\rm Tr} \bigl(   (D^2 \Phi)_v   \cdot  Dv  \cdot  (D^2 \Phi)^{-1} \bigr)  \eta  \ d \mu
+ \int \mbox{\rm Tr} \Bigl[ (D^2 \Phi)^{-1} (D^2 \Phi)_v \bigr]^2 \eta \ d\mu.
\end{align*}

\end{lemma}
\begin{proof}
It follows from the change of variables formula 
$V = W(\nabla \Phi) - \log {\det}  D^2 \Phi$
that
\begin{align*}
V(x+ tv) - V(x)  & =
 W(\nabla \Phi(x+tv)) - W(\nabla \Phi(x))
\\& - \log \Bigl[ ({\det} D^2 \Phi(x))^{-1}  \cdot {\det} D^2\Phi(x+tv)\Bigr].
\end{align*}
By the $K$-uniform convexity of $W$
\begin{align*}
 W(\nabla \Phi(x+tv)) - W(\nabla \Phi(x))  &
\ge
\langle \nabla \Phi(x+tv)- \nabla \Phi(x), \nabla W(\nabla \Phi(x))  \rangle
\\& + \frac{K}{2} |\nabla \Phi(x+tv)- \nabla \Phi(x)|^2.
\end{align*}
This implies
\begin{align*}
\int (V(x+ tv) - V(x))  \eta \ d\mu &
\ge
\int \frac{K}{2} |\nabla \Phi(x+tv)- \nabla \Phi(x)|^2  \eta \ d\mu
\\&
+ \int  \langle \nabla \Phi(x+tv)- \nabla \Phi(x), \nabla W(\nabla \Phi(x))  \rangle  \eta \ d\mu
\\& - \int \log \Bigl[ ({\det} D^2 \Phi(x))^{-1}  \cdot {\det} D^2\Phi(x+tv)\Bigr] \eta  \ d \mu.
\end{align*}
Denote by $\Psi=\Phi^*$ the convex conjugated function of $\Phi$.
Using the fact that $\nabla \Psi$ and $\nabla \Phi$ are reciprocal we get
\begin{align*}
  \int  \langle \nabla \Phi(x+tv)& - \nabla \Phi(x), \nabla W(\nabla \Phi(x))  \rangle  \eta  \ d\mu
\\&  = \int  \langle \nabla \Phi(x+tv)\circ (\nabla \Psi)- x, \nabla W(x)  \rangle  \eta(\nabla \Psi) \ d\nu
\\&
=
\int \mbox{\rm{div}} \bigl(  \nabla \Phi(x+tv)\circ (\nabla \Psi)- x\bigr)  \eta(\nabla \Psi) e^{-W} \ dx
\\& +
\int  \langle \nabla \Phi(x+tv)\circ (\nabla \Psi)- x, D^2 \Psi \cdot \nabla \eta(\nabla \Psi)  \rangle \ d\nu
.
\end{align*}
By the relation $ (D^2 \Phi(\nabla \Psi) )^{-1} = D^2\Psi$  we get
\begin{align*}
& \int \mbox{\rm{div}} \bigl(  \nabla \Phi(\nabla \Psi(x)+tv )- x\bigr)  \eta(\nabla \Psi) e^{-W} 
\\&
=
\int \bigl( \mbox{\rm Tr} D^2 \Phi (x+ tv) \circ(\nabla \Psi) \cdot (I + t Dv) \circ (\nabla \Psi)   \cdot  D^2 \Psi -  d \bigr) \eta(\nabla \Psi) \ d\nu
\\&
=\int \bigl( \mbox{\rm Tr} D^2 \Phi(x+ tv)   \cdot (I + t Dv)  \cdot  (D^2 \Phi(x))^{-1}-  d \bigr) \eta \ d \mu.
\end{align*}
Remark that
$$
\mathcal{T}(x,tv) = 
  \mbox{\rm Tr} D^2 \Phi(x+ tv)   \cdot  (D^2 \Phi(x))^{-1}-  d
-  \log \Bigl[ {\det} D^2 \Phi(x)  \cdot ({\det} D^2\Phi(x+tv))^{-1}\Bigr]
\ge 0.
$$
Thus one obtains
\begin{align*}
\int & (V(x+ tv) - V(x))  \eta \ d\mu 
\ge
\frac{K}{2} \int  |\nabla \Phi(x+tv)- \nabla \Phi(x)|^2 \eta \ d\mu
\\&
+ t \int \mbox{\rm Tr} \bigl(  D^2 \Phi(x+ tv)   \cdot  Dv(x)  \cdot  (D^2 \Phi(x))^{-1} \bigr) \eta \ d \mu
\\&
+
\int  \langle \nabla \Phi(x+tv) - \nabla \Phi(x), (D^2 \Psi) \nabla \eta(\nabla \Psi(x))  \rangle \ d\nu 
+ \int \mathcal{T}(x,tv) \eta \ d \mu.
\end{align*}
Now apply the same inequality to $-tv$, take the sum, and divide by $t^2$. 
It can be easily verified with the help of the Taylor formula that
$$
\lim_{t \to 0} \frac{\mathcal{T}(x,tv)  + \mathcal{T}(x,-tv)}{t^2} = \mbox{Tr} \Bigl[ (D^2 \Phi)^{-1} (D^2 \Phi)_v \bigr]^2 \ge 0.
$$
In the limit $t \to 0$ one gets the desired inequality.
\end{proof}

The proof of the Lemma \ref{zero-disc} follows some elementary measure-theoretical arguments and we omit it here.
It relies on the fact that the set of symmetric nonnegative matrices with multiple eigenvalue has smaller dimension in the ambient space of all symmetric nonnegative matrices.

\begin{lemma}
\label{zero-disc}
Assume that $\Phi$ is convex and twice continuously differentiable. For every $\varepsilon > 0$ there exists a matrix $Q_{\varepsilon} \ge 0$
such that   $\|Q_{\varepsilon}\| \le \varepsilon$ and $D^2 \Phi +  Q_{\varepsilon}$  has no multiple eigenvalues almost everywhere.
\end{lemma}

\begin{theorem}
\label{norm-result} Assume that $D^2 W \ge K \cdot \mbox{\rm Id}$ and $(D^2 V)_{+} \in L^1 (\mu)$.
Then the following inequality holds
$$
\int  \| (D^2 V)^+ \|  \  d \mu   \ge K \int \| D^2 \Phi\|^2   \  d \mu. 
$$
\end{theorem}
\begin{proof}

{\bf Step 1.}
Let $\Phi$ be smooth.
Fix a point $x_0$. Assume that $D^2 \Phi(x_0)$ has no multiple eigenvalues.
Assume that $v$ is a smooth field coinciding with the unit eigenvectors of $D^2 \Phi$ corresponding to the unique largest eigenvalue $\lambda$ 
in a neighborhood $U_{x_0}$ of $x_0$.
Let us show that   $\mbox{\rm Tr} \bigl(  \partial_v D^2 \Phi   \cdot  Dv  \cdot  (D^2 \Phi)^{-1} \bigr)  \ge 0$ in $U_{x_0}$.

Indeed, one has 
$$
D^2 \Phi \cdot v= \lambda \cdot v, \ \  |v|=1.
$$
Differentiating both identities we get
$$
(D v )^{T} v =0,
$$
$$
\partial_v D^2 \Phi + D^2 \Phi \cdot Dv = \lambda \cdot Dv + v \oplus \nabla \lambda.
$$
Multiply (from the left) the second identity by $(D^2 \Phi)^{-1} \cdot (Dv)^{T}$ and take the trace.
Taking into account that
$$
\mbox{Tr} (D^2 \Phi)^{-1} (Dv)^{T} \cdot v \oplus \nabla \lambda  = \langle (D^2 \Phi)^{-1} (Dv)^{T}  v, \nabla \lambda \rangle =0
$$
one obtains
$$
\mbox{Tr} (D^2 \Phi)^{-1} (Dv)^{T} \cdot \partial_v D^2 \Phi + \mbox{Tr} (D^2 \Phi)^{-1} (Dv)^{T}  D^2 \Phi \cdot Dv
=
\lambda \cdot \mbox{Tr}  (D^2 \Phi)^{-1} (Dv)^{T} \cdot Dv.
$$
Finally we get
\begin{align*}
 \mbox{\rm Tr} \bigl(  \partial_v D^2 \Phi   \cdot  Dv  \cdot  (D^2 \Phi)^{-1} \bigr) &  =
\mbox{Tr} (D^2 \Phi)^{-1} (Dv)^{T} \cdot \partial_v D^2 \Phi 
\\&
=  \mbox{Tr}  (D^2 \Phi)^{-1} (Dv)^{T} (\lambda I - D^2 \Phi)\cdot Dv
\\&
= \mbox{Tr}  (D^2 \Phi)^{-1/2} (Dv)^{T} (\lambda I - D^2 \Phi) \cdot Dv \cdot (D^2 \Phi)^{-1/2}.
\end{align*}
Note that the latter is equal to
$$
\mbox{Tr} (A B A^{T}), \mbox{where} \  \  A = (D^2 \Phi)^{-1/2} (Dv)^{T} (D^2 \Phi)^{-1/2} , \   \   B = \lambda D^2 \Phi - (D^2 \Phi)^2.
$$
Since $\lambda$ is the largest eigenvalue,  $B$ is symmetric and non-negative. This immediately implies that  $$\mbox{\rm Tr} \bigl(  \partial_v D^2 \Phi   \cdot  Dv  \cdot  (D^2 \Phi)^{-1} \bigr) \ge 0$$

In particular, if $\mbox{supp}(\eta) \subset   U_{x_0} $, we obtain from the previous lemma
\begin{align}
\label{local-eig-est}
\int \| (D^2 V)^+ \|^2 \eta \  d \mu & 
\ge K \int \lambda^2  \cdot \eta  \  d \mu
+  \int \langle (D^2 \Phi)_v \cdot v , (D^2 \Phi)^{-1}\nabla \eta  \rangle  \ d\mu 
\nonumber
\\& + \int \mbox{\rm Tr} \Bigl[ (D^2 \Phi)^{-1} (D^2 \Phi)_v \bigr]^2 \eta \ d\mu.
\end{align}

{\bf Step 2.} 
Let us assume  that $\Phi$ is a convex polynom such that $D^2 \Phi$
has no multiple eigenvalues almost everywhere.
Recall that the set $S:=S(D^2 \Phi)$, where $D^2 \Phi$ has multiple eigenvalues, is the zero set of the discriminant of 
$D^2 \Phi$. Hence $S$ is an algebraic variety. In particular, for $\mathcal{H}^{d-1}$-almost every point $x \in \partial S$ the set
$S \cap B_r(x)$ is diffeomorphic to $\R^{d-1}$ for sufficiently small $r$ (see \cite{BCR}, Proposition 3.3.14).  Let $\R^d \setminus S = \cup D_i$, where every $D_i$ is a connected component of $\R^d \setminus S$. Clearly, one can choose a vector  field of unit eigenvectors $v$ corresponding to the 
largest eigenvalue of $D^2 \Phi$ such that $v|_{D_i}$ is smooth for every $D_i$.

By a classical result of Ky Fan \cite{KyFan} the function $A \to \Lambda(A)$, where 
$\Lambda(A)$ is the largest eigenvalue is convex on the set of symmetric matrices.
This implies, in particular, that the function
$$
\lambda(x) : x \to \Lambda(D^2 \Phi(x))
$$
has a directional derivative
$$
\partial_e \lambda(x) = \lim_{t \to 0} \frac{\lambda(x+te) - \lambda(x)}{t}
$$
for every $x$ and every direction $e$.
For every regular point $x \in \partial D_i$ we define
$$
\nabla_{D_i} \lambda = \sum_{i=1}^d \partial_{e_i} \lambda \cdot e_i,
$$
where the basis $\{e_i\}$ is chosen in such a way that $x + te_i \in D_i$ for the small values of $t$ and every $i$.

Note that
$$\partial_e \langle D^2 \Phi \cdot v, v \rangle = \langle  D^2 \Phi_e \cdot v, v \rangle + 2 \langle D^2 \Phi \  \partial_e v, v \rangle$$ 
inside of $D_i$,
Since $v$ is an unit eigenvalue of $D^2 \Phi$, $\partial_e v$ is orthogonal to $D^2 \Phi v$ and one has
$$
\partial_e \lambda(x) = \langle  D^2 \Phi_e \cdot v, v \rangle.
$$

Let us fix a compact domain $B$ with smooth boundary and  apply (\ref{local-eig-est}) to $\nu = I_{D_i \cap B}$. More precisely, we choose a sequence of smooth test functions
$\{\eta_n\}$ with supports inside of $D_i \cap B_R$ such that $\eta_i \to I_{D_i \cap B} $.
One gets in the limit
\begin{align}
\label{loc-eig-est}
\int_{D_i \cap B} & \| (D^2 V)^+ \|^2 \  d \mu  
\ge K \int_{D_i \cap B} \lambda^2   \  d \mu 
+
\int_{D_i \cap B} \mbox{\rm Tr} \Bigl[ (D^2 \Phi)^{-1} (D^2 \Phi)_{v} \Bigr]^2 \ d\mu
\\&
\nonumber
+   \int_{\partial D_i \cap B } \big\langle  \nabla_{D_i} \lambda, (D^2 \Phi)^{-1} n_{D_i} \big\rangle  \  d\mu +  
\int_{D_i \cap \partial B} \big\langle \nabla_{D_i} \lambda, (D^2 \Phi)^{-1} n_B \big\rangle  \ d\mu,
\end{align}
where $n_B$ is the inward  normal to $\partial B$.

Now take a regular point $x \in \partial D_i$. Clearly, $x$ belongs to the border between two sets $D_i$ and $D_j$, $j \ne i$
and the inward normal of $\partial D_i$
can be computed in the following way
$$
n_{ij} = \frac{\nabla_{D_i} \lambda - \nabla_{D_j} \lambda}{\| \nabla_{D_i} \lambda - \nabla_{D_j} \lambda \|}.
$$
Taking the sum of (\ref{loc-eig-est}) over $i$ we get
that the integral term over the boundary $\cup_{i} \partial D_i \cap B$ takes the form
$$
\sum_{i,j}  \int_{\partial D_i \cap \partial D_j \cap B } \Big\langle \nabla_{D_i} \lambda -  \nabla_{D_j}  \lambda,   (D^2 \Phi )^{-1} n_{ij}  \Big\rangle  \  d\mu 
$$
and it is obviously non-negative.

Taking the sum over $i$ we get
\begin{align*}
\int_{ B} & \| (D^2 V)^+ \|^2 \  d \mu  
\ge K \int_{ B} \lambda^2   \  d \mu 
+
\int_{B} \mbox{\rm Tr} \Bigl[ (D^2 \Phi)^{-1} (D^2 \Phi)_{v} \Bigr]^2 \ d\mu
\\&
+ \sum_i \int_{D_i \cap \partial B} \big\langle \nabla_{D_i} \lambda, (D^2 \Phi)^{-1} n_B \big\rangle  \ d\mu.
\end{align*}

Fix a smooth 
compactly supported nonnegative test function $\xi$.
Applying the coarea formula and the above estimate applied to the level sets  of $\xi$ one can easily get that 
\begin{align*}
\int & \| (D^2 V)^+ \|^2 \xi \  d \mu  
\ge K \int \lambda^2 \xi   \  d \mu 
+
\int \mbox{\rm Tr} \Bigl[ (D^2 \Phi)^{-1} (D^2 \Phi)_{v} \Bigr]^2 \xi \ d\mu
\\&
+  \sum_i  \int_{D_i} \langle (D^2 \Phi)_{v} \cdot v , (D^2 \Phi)^{-1} \nabla \xi \rangle  \ d\mu .
\end{align*}
Applying the standard relations between the operator and Hilbert-Schmidt norms
\begin{align*}
\mbox{\rm Tr} \Bigl[ (D^2 \Phi)^{-1} (D^2 \Phi)_{v} \Bigr]^2 &
=
\| (D^2 \Phi)^{-1/2} (D^2 \Phi)_{v}  (D^2 \Phi)^{-1/2}\|_{HS}^2  
\\&
\ge \frac{ \|(D^2 \Phi)_{v}\|^2_{HS}}{\|D^2 \Phi \|^2}
 \ge \frac{ \|(D^2 \Phi)_{v}\|^2}{\|D^2 \Phi \|^2}.
\end{align*}
and the Cauchy inequality one finally gets
\begin{equation}
\label{27.07}
\int  \| (D^2 V)^+ \|^2 \xi \  d \mu   + 4 \int \frac{|\nabla \xi|^2}{\xi} \ d \mu
\ge K \int \lambda^2 \xi   \  d \mu. 
\end{equation}
Choosing an appropriate sequence of compactly supported functions $\{\xi_n\}$
such that $\lim_n \xi_n =1$ and $\lim_{n} \int  \frac{|\nabla \xi_n|^2}{\xi_n} \ d\mu =0$ we get the claim.

{\bf Step 3}. Here we prove the general case. 
In the same way as in Theorem \ref{mr} one can approximate $V$ and $W$ by smooth functions with at most quadratic growth.
 Hence, one can assume without loss of generality that $\Phi$ is smooth.
To apply the previous step we fix a compact set $B$
and choose a sequence of polynomial functions $\{\Phi_n\}$ such that $\Phi_n \to \Phi$ on $B$  locally uniformly with all
the derivatives up to the fourth order (this can be done by a multidimensional version of
the Weierstrass approximation theorem).

Since we have convergence of the second derivatives, the functions $\Phi_n$ are convex for  sufficiently big $n$.
Applying Lemma \ref{zero-disc} we may assume that $S(\Phi_n)$ has zero measure.
Note that
the mapping $\nabla \Phi_n$ sends $e^{-V_n} \ dx$ onto $\mu$, where 
$$
V_n = W(\nabla \Phi_n) - \log \det D^2 {\Phi_n}.
$$
From the convergence $\Phi_n \to \Phi$ follows that $V_n \to V$ uniformly in $B$ and the same holds for the derivatives up to the  second order.
Passing to the limits one obtains (\ref{27.07}) for $V$ and any smooth test function $\xi$ . Choosing an appropriate sequence $\{\xi_n\}$ with ${\xi_n} \to 1$ 
one can easily complete the proof.
\end{proof}

The  following result  generalizes  Theorem \ref{norm-result} in the same
manner as Theorem \ref{lp-est} generalizes Theorem \ref{mr}.
The proof can be obtained by modifying the  proof of  Theorem \ref{norm-result} and we omit it here.

\begin{theorem}
 Assume that $D^2 W \ge K \cdot \mbox{\rm Id}$. Then for every $r \ge 1$ one has
$$
K 
\Bigl( \int  \| D^2 \Phi\|^{2r} \ d \mu \Bigr)^{\frac{1}{r}}
\le \Bigl(  \int  \| (D^2 V)_{+} \|^r  \ d \mu \Bigr)^{\frac{1}{r}}.
$$
\end{theorem}

This work was supported by the RFBR projects 07-01-00536 and 08-01-90431-Ukr, the  DAAD Grant (2010),
and the SFB701 at the University of Bielefeld.

\end{document}